\documentclass[12pt]{amsart}
\usepackage{amsmath, amsfonts, amsthm, amssymb, multicol, mathtools, dsfont, verbatim, enumerate}
\usepackage{graphicx}
\usepackage{float, hyperref}
\usepackage{scalerel,stackengine, subcaption}
\usepackage[usenames,dvipsnames]{xcolor}
\usepackage{enumitem}
\usepackage{pgfplots}
\usepgfplotslibrary{fillbetween}
\pgfplotsset{width=10cm,compat=1.9}

\usepackage[square,sort,comma,numbers]{natbib}
\allowdisplaybreaks

\usepackage{fancyhdr}
 
\numberwithin{equation}{section}

\newtheorem{thm}{Theorem}[section]
\newtheorem{lma}[thm]{Lemma}

\newtheorem{defn}[thm]{Definition}

\newtheorem{prop}[thm]{Proposition}

\newtheorem{rem}[thm]{Remark}

\renewcommand{\epsilon}{\varepsilon}

\renewcommand{\geq}{\geqslant}
\renewcommand{\leq}{\leqslant}

\renewcommand{\epsilon}{\varepsilon}

\renewcommand{\geq}{\geqslant}
\renewcommand{\leq}{\leqslant}

\makeatletter
\def\@setauthors{%
  \begingroup
  \def\thanks{\protect\thanks@warning}%
  \trivlist
  \centering\footnotesize \@topsep30\p@\relax
  \advance\@topsep by -\baselineskip
  \item\relax
  \author@andify\authors
  \def\\{\protect\linebreak}

  \normalsize\lowercase{\authors}%
  
	\ifx\@empty\contribs
  \else
    ,\penalty-3 \space \@setcontribs
    \@closetoccontribs
  \fi
  \endtrivlist
  \endgroup
}
\def\@settitle{\begin{center}
\LARGE\lowercase{\@title}
  \end{center}%
}
\makeatother
\newcommand{\authoremail}[1]{\email{\href{mailto:#1}{\color{lightblue}{#1}}}}
\newcommand{\authoraddress}[1]{\address{\normalfont{#1}}}

\definecolor{lightblue}{HTML}{2B77A4}
\definecolor{darkred}{HTML}{9E0D0D}
\hypersetup{
	colorlinks=true,
	linkcolor=darkred,
	urlcolor=darkred,
	citecolor=lightblue
}
\title{ On the Onsager's energy conservation for the convergence dynamics of the 3D-Leray-$\alpha$ gaseous star model}

\author{ANIS RAHMANI}
\authoraddress{Former-Student of Department of Mathematics, University of Batna 2, Mostefa Ben
Boulaid, Fesdis, Batna 05078, Algeria.}
\authoremail{rahmanianis58@gmail.com}

\author{ABDELAZIZ MENNOUNI}
\authoraddress{Department of Mathematics, LTM, University of Batna 2, Mostefa Ben Boulaid, Fesdis, Batna 05078, Algeria.}
\authoremail{a.mennouni@univ-batna2.dz}

\date{}

\begin{document}


\maketitle
\thispagestyle{empty}

\begin{abstract}
This research presents a model that accurately represents the motions of gaseous stars We employ the Navier-Stokes-Poisson system to transform compressible Euler equations into non-compressible ones by combining quasineutral and inviscid conditions. We intend to put Onsager's hypothesis to the test using the Leray-alpha gaseous star model. This conjecture connects energy conservation with the regularity of weak solutions in the Euler equations. The model used in this work functions as an inviscid regularization of the Euler equations. It technically converges to the Euler equations as the regularization length scale $\alpha$ approaches $0^{+}$.
\end{abstract}

\tableofcontents

\section{Introduction}
Nonlinear partial differential equations are used to explain scientific contexts in several domains (eg.,\cite{1} \cite{2} \cite{3} \cite{4}).
 Navier-Stokes-Poisson system, plays an important role in the framework of preserving mass in cosmic dynamics, carefully governs the development of celestial entities. The dual-species chemical kinetics equations maintain momentum-entropy equilibrium, whereas the Poisson equation defines the gravitational potential $\psi$:
\begin{equation}\label{sys1}
\left\{
\begin{aligned}
&\partial_{t}\varrho + div(\varrho\textbf{u})=0,\\
     &\partial_{t}(\varrho\textbf{u})+div(\varrho\textbf{u}\otimes \textbf{u})+\nabla_{x}p=\mu \Delta \textbf{u}+(\upsilon+\mu )\nabla_{x} div\textbf{u}-\varrho\nabla_{x}\psi,
     \\ 
     &\partial_{t}(\varrho s)+div(\varrho s \textbf{u})+div(\dfrac{\textbf{q}}{\theta})= \sigma,
     \\ 
      &\partial_{t}(\varrho Z)+div(\varrho Z \textbf{u})=w(\varrho,\theta,Z)+div(\mathcal{F}),
      \\ 
     &-\lambda^{2} \Delta \psi=\varrho-1. 
			\end{aligned}
\right.
\end{equation}
The model proposed for gaseous star dynamics is based on the Navier-Stokes-Poisson system formulation, which is specifically intended for compressible and reactive gases. This concept explains the complex interplay between fluid dynamics and thermodynamics in a gaseous star environment.
\par Let s denote the specific entropy, and p stand for pressure. $\textbf{q}$ denotes the heat flux, and $\sigma$ denotes the entropy production rate. The function $w=w(\varrho,\theta,Z)$ represents the rate of nuclear reactions occurring within the stellar core and $\mathcal{F}$ denotes the species diffusion flux. Furthermore, $\lambda$ represents the (scaled) Debye length, while $\mu$ and $\upsilon$ stand as the constant viscosity coefficients in satisfying conditions.
\begin{eqnarray}\label{sys2}
    \mu > 0\;\;\mbox{and}\;\; \upsilon + \dfrac{2}{n} \mu \geq 0.
\end{eqnarray}
Furthermore, we take into account the following initial conditions:
\begin{equation}\label{sys3}
\left\{
\begin{aligned}
      &\varrho(0,x)=\varrho_{0}, \\
      &\textbf{u}(0,x)=\textbf{u}_{0}(x), \\
      &\theta(0,x)=\theta_{0}(x),\\
      &Z(0,x)=Z_{0}(x),\\
\end{aligned}
\right.
\end{equation}
The function $w=w(\varrho,\theta,Z)$, representing the rate of nuclear reactions, can exhibit a broad generality. A concrete example is given by:
$w=K\phi(\theta),\;\;\varrho Z=0.$
The function $w=w(\varrho,\theta,Z)$, represents the rate of nuclear reactions, which can exhibit a broad generality. A concrete example is given by:
$w=K\phi(\theta) \varrho Z=0.$
The inclusion of the parameter $K$ in the fourth chemical kinetics equation indicates its role as the reaction rate parameter. $\phi=\phi(\theta)$ represents the reaction function governed by the Arrhenius kinetics. Also, the reaction function $\phi$ is assumed to be an increasing, Lipschitz continuous function on $[0,\infty)$ satisfying
\begin{equation}\label{sys4}
\left\{
\begin{aligned}
&\phi(\theta)=0,\quad 0\leq \theta \leq \theta_{i}, \\
&\phi(\theta)>0,\quad \theta > \theta_{i},\\
\end{aligned}
\right.
\end{equation}
where $\theta_{i} > 0$ represents the ignition temperature in this situation. This criterion is met, as shown by the Arrhenius function $ \phi(\theta) = e^{-A/\theta}$, $\theta \gg \theta_{i}$. As a result, a fluid particle is initially made up entirely of the reactant. When the ignition temperature is exceeded, combustion occurs, converting a portion or all of the mass into the product species.
\par The rate of heat transfer, shown by the heat flux $\textbf{q}$, is directly related to the difference in temperature across space. It follows the Fourier law and has a heat conductivity coefficient $k \in C^{2}[0,\infty]$ that changes with temperature $\theta $ in the following way:
\begin{equation}\label{sys4}
\left\{
\begin{aligned}
   \mbox{\textbf{Fourier’s law:}}\quad &\textbf{q}=-k\nabla_{x} \theta,\;\; k>0,\\
   & c_{1}(1+\theta^{3}) \leq k(\theta) \leq  c_{2}(1+\theta^{3}).\\  
\end{aligned}
\right.
\end{equation}
Let $\mathcal{F}$ stand for the species' diffusion flux, which has a linear relationship with the spatial gradient of the state variable $Z$. Its expression is given by the equation and follows Fick's law, where $d>0$ is the species' diffusion coefficient:
\begin{eqnarray*}
    \mbox{\textbf{Fick’s law:}}\quad\mathcal{F}=d\nabla_{x} Z.
\end{eqnarray*}
The model analysis focuses on the combined quasineutral and inviscid limit of the Navier-Stokes-Poisson system within the torus $\mathbb{T}^{3}$. This limit helps to explain the system's behavior under quasi-neutrality (where positive and negative charges virtually cancel each other out) and minimal viscosity.
\par We demonstrate that the Navier-Stokes-Poisson system converges to the incompressible Euler equations. We establish convergence for both the global weak solution and the scenario with broad initial data. Such findings lead to a better understanding of the system's behavior and its relationship to the incompressible Euler equations, revealing important information on the dynamics of gaseous stars.
As in \cite{5}, the system $(\ref{sys5})$ is an approach of the incompressible Euler equations of ideal reactive fluids in the unknowns $(\textbf{u},p,Z)$, given by:
\begin{equation}\label{sys5}
\left\{
\begin{aligned}
    &\partial_{t} \textbf{u} + (\textbf{u}\cdot\nabla)\textbf{u} +\nabla p = 0,\quad \quad\quad (x,t)\in \mathbb{T}^{3}\times (0,T)\\
    &\partial_{t}Z + (\textbf{u}\cdot\nabla)Z + K\phi(\Bar{\theta})\Bar{Z}=0,\quad \quad\quad (x,t)\in \mathbb{T}^{3}\times (0,T)\\
&-K\phi(\Bar{\theta})\Bar{Z}=0,\quad \quad\quad (x,t)\in \mathbb{T}^{3}\times (0,T)\\
&\nabla \cdot\textbf{u}=0 \quad \quad\quad (x,t)\in \mathbb{T}^{3}\times (0,T).
\end{aligned}
\right.
\end{equation}
Taking the combined quasineutral and inviscid limit $( \lambda,\mu,d,\upsilon)\rightarrow 0$ and $\varrho=1$
This means that the system $(\ref{sys1})$ is approximated by a system that generalizes the incompressible Navier-Stokes equation for a viscous reacting fluid in the real world. Formally, letting the viscosity parameters $\mu$, $\upsilon\rightarrow0$ and the species diffusion $d \rightarrow0$ yield the system $(\ref{sys5})$
\begin{prop}(\cite{5,22})
    Let $\textbf{u} $ be a divergence-free vector field in $L^{\infty}(([0,T];L^{2}(\mathbb{T}^{3}))$ and $V=(\nabla_{x}q,\nabla_{x}\phi)^{T}$ is an element of $L^{\infty}(([0,T];L^{2}(\mathbb{T}^{3}))$. Then
    \begin{eqnarray*}
        \mathcal{L}(-\dfrac{t}{\epsilon})A^{\epsilon}(t) \rightarrow \Bar{B}(\nu,\nu)\;\;\mbox{weakly}-^{*}\;\;\mbox{in} L^{\infty}(([0,T];W^{-1,1}(\mathbb{T}^{3}))\;\;\mbox{as}\;\;\epsilon \rightarrow 0.
    \end{eqnarray*}
\end{prop}
\begin{prop}(\cite{5,22})
    Let \textbf{u} be a divergence -free vector fields $L^{r}(0,T;H^{s}({\mathbb{T}^{3}} ))$ and $V=(\nabla_{x}q,\nabla_{x}\phi)^{T}$ is an element of $L^{r}(0,T;H^{s}({\mathbb{T}^{3}} ))$. Let $w^{\epsilon}$ be a sequence of divergence-free vector fields $w^{\epsilon}$ $  L^{p}(0,T;H^{-s}({\mathbb{T}^{3}} ))$ and $W^{\epsilon}=(\nabla_{x}g^{\epsilon},\nabla_{x}\psi^{\epsilon})^{T}\in L^{p}(0,T;H^{-s}({\mathbb{T}^{3}} ))$ such that  $\dfrac{1}{r}+\dfrac{1}{p}=1$. Assume that there exist $w$ and $\Bar{X}=(\nabla_{x}g,\nabla_{x}\psi)^{T}$ such that $w^{\epsilon} $ and $X^{\epsilon}=\mathcal{L}(-\frac{t}{\epsilon})W^{\epsilon}$ converges strongly to $w$ and $\Bar{X} \in   L^{p}(0,T;H^{-s}({\mathbb{T}^{3}} ))$ for all $s'\geq s > \frac{n}{2}+2$, respectively. Then 
    \begin{eqnarray*}
        \mathcal{L}(-\dfrac{t}{\epsilon})A^{\epsilon}(t) \rightarrow \Bar{B}(\Bar{\mathcal{Z
        }},\nu)\;\;\mbox{as}\;\;\epsilon\rightarrow 0.
    \end{eqnarray*}
\end{prop}
\begin{rem}\label{proposition 3}
    If  $p=2$ and $n=3$, we obtain 
		$$
		s'\geq s > \frac{7}{2},\;\;\mbox{and}\;\;\textbf{u}\in L^{2}(0,T;H^{s}({\mathbb{T}^{3}} )),\;\; w \in L^{2}(0,T;(H^{-s'}({\mathbb{T}^{3}} )).
		$$
Thus 
$$
\textbf{u} \in L^{\infty}(0,T;L^{2}({\mathbb{T}^{3}} )) \cap  L^{2}(0,T;H^{s}({\mathbb{T}^{3}})),\;\;\mbox{with}\;\; s>\frac{7}{2}.
$$
\end{rem}
Set
$$
\mathcal{V}= \left\{\varphi / \varphi\in \mathcal{D}(\mathbb{T}^{3}), div \varphi = 0\right\} .
 $$
  Let $\textbf{H}$ be the adhesion of $ \mathcal{V}$ in $ L^{2}(\mathbb{T}^{3}) $,
$\textbf{V}_{s}$ adhesion of $ \mathcal{V}$ in $ H^{s}(\mathbb{T}^{3}) $ in particular
$\textbf{V}_{1}= \textbf{V} $ adhesion of $ \mathcal{V}$ in $ H^{1}(\mathbb{T}^{3})$. According to \cite{12}, we have
\begin{eqnarray*}
    \textbf{V}_{s}\subset\textbf{V}\subset \textbf{H},\;\;s>1.
\end{eqnarray*}
Note that 
$$
\textbf{u}\in L^{\infty}(0,T;L^{2}({\mathbb{T}^{3}} )) \cap  L^{2}(0,T;H_{0}^{1}({\mathbb{T}^{3}})).
$$
\section{History of Onsager's Conjecture}
In \cite{11}, Onsager discussed the regularity threshold for a weak solution $u$ to the Euler equations, which is $C_{t}^{0}L^{2}$. This regularity is critical for assuring the complete elimination of the whole energy flux, which facilitates energy conservation. Onsager's first conjecture, which is now a theorem, can be expressed in contemporary mathematical terms as follows:
For $\alpha\in(0,1)$, let $\textbf{X}_{\infty}^{\alpha}=C_{t}^{0}C^{\alpha}$ be the space of continuous functions in time that are $\alpha$-Holder continuous. in space.   Also, define $\textbf{X}_{3}^{\alpha}=L_{t}^{3}B_{3,\infty}^{\alpha}$ the space of functions such that the Besov norm 
\begin{eqnarray}
  ||u||_{B_{3,\infty}^{\alpha}}:= ||u||_{L^{3}} + \sup_{|y|>0} \dfrac{||u(.+y)-u(.)||_{L^{3}}}{|y|^{\alpha}}
\end{eqnarray}
is cubically integrable in time. Onsager outlined this contrast in \cite{11}.
\begin{thm}{(The Onsager Theorem)}
\begin{enumerate}
	\item If $\alpha > \frac{1}{3}$. Then any weak solution $u\in C_{t}^{0}L^{2}\cap \textbf{X}_{\infty}^{\alpha}$ conserves its
kinetic energy. The same is true when  $\textbf{X}_{\infty}^{\alpha} $ is replaced by $\textbf{X}_{3}^{\alpha}$.
\item  If $\alpha\leq \frac{1}{3}$. Then, there exist weak solutions $u\in C_{t}^{0}L^{2}\cap \textbf{X}_{\infty}^{\alpha}$ which
dissipate kinetic energy; i.e., the kinetic energy is a non increasing
function of time. The same is true when $\textbf{X}_{\infty}^{\alpha} $ is replaced by $\textbf{X}_{3}^{\alpha}$.
\end{enumerate}
 \end{thm}
Some outcomes of the dichotomy are established in \cite{15} and later in \cite{10}.
\par Denote by 
$$
\mathbb{P}_{\leq\kappa}(u\otimes u)-\mathbb{P}_{\leq\kappa} u\otimes \mathbb{P}_{\leq\kappa} u\;\;\mbox{such that}\;\;\mathbb{P}_{\leq\kappa}
$$
the operator that truncates the Fourier modes, which have an absolute value larger than $\kappa$. That is, 
$$
\mathbb{P}_{\kappa} u(x) = \sum\limits_{k\in \mathbb{Z}^{3},|k|<\kappa}\hat u (x) e^{ i k\cdot x}.
$$
Define the flux density of a periodic vector field $u$ at frequencies of modulus $\kappa$ as a function of $(t, x)$:
\begin{eqnarray*}
\pi_{\kappa}[u]=(\mathbb{P}_{\leq\kappa}(u\otimes u)-\mathbb{P}_{\leq\kappa} u\otimes \mathbb{P}_{\leq\kappa} u):\;\; \nabla\mathbb{P}_{\leq\kappa} u.
\end{eqnarray*}
When referring to matrices $A$ and $B$, the expression $A:\;B$ indicates contraction $\sum\limits_{i,j} A_{ij} B_{ij}$.
Ultimately, we estimate the flow of frequencies with a magnitude of $\kappa$ across time $(t)$
\begin{eqnarray*}
    \Pi_{\kappa}[u]=\int_{\mathbb{T}^{3}}\pi_{\kappa}[u](.,x)dx.
\end{eqnarray*}
Recently, Constantin and Titi established the following bound 
\begin{eqnarray*}
    || (\mathbb{P}_{\leq\kappa}(u\otimes u)-\mathbb{P}_{\leq\kappa} u\otimes \mathbb{P}_{\leq\kappa} u)||_{L^{\frac{3}{2}}} \leq \kappa^{-2\alpha} ||u||_{B_{3,\infty}^{s}}^{2}.
\end{eqnarray*}
This can be combined with
$$
|| \nabla\mathbb{P}_{\leq\kappa} u||_{L^{3}} \leq \kappa^{1-\alpha} ||u||_{B_{3,\infty}^{s}}
$$
to yield
\begin{eqnarray}\label{20}
    |\Pi_{\kappa}[u](t)|\precsim \kappa^{1-3\alpha}||u(t,.)||_{B_{3,\infty}^{s}}^{3}.
\end{eqnarray}
The bound $(\ref{20})$ demonstrates that $u\in L_{t}^{3}B_{3,\infty}^{\alpha}$ with $\alpha > \frac{1}{3}$. The amount $\Pi_{\infty}[u,t',t]$ indicates that the total energy flow of the function u over the time span $[t',t]$ is zero, hence proving the energy equality:
\begin{eqnarray*}
 E(t)-E(t')=-\lim\limits_{\kappa\rightarrow \infty} \int_{t'}^{t}\Pi_{\kappa}[u](s) ds =: -\Pi_{\infty}[u,t',t].
\end{eqnarray*}
The work of Constantin and Titi was improved by Duchon and Robert in \cite{7}, who introduced the energy dissipation measure $D[u]$ as defined earlier in equation $(\ref{40})$. They also established the significant identity $D[u]=\lim\limits_{\kappa \rightarrow\infty} \Pi_{\kappa}[u]$ in the context of distributions, provided that $u\in L_{t}^{3}L^{3}$ is a weak solution of the Euler equations.
\par In \cite{14}, Eyink extensively examines the concept of flux locality. It is known that the threshold regularity in \cite{7} and \cite{16} is sharp because it is found in the 1D Burgers equation, which shows where a classical shock loses energy. It is in $B_{p,\infty}^{\frac{1}{p}}$ for all $p\in [1, \infty]$.and so in particular in $B_{3,\infty}^{\frac{1}{3}}$, but it does not lie in the space $B_{3,c_{0}}^{\frac{1}{3}}$. Lastly, we highlight the work \cite{13} by Shvydkoy, which explores various geometric constraints that guarantee the conservation of energy for the threshold value. $\alpha =\frac{1}{3}$.
Duchon and Robert improved Constantin, E., and Titi's work (cite 3). They came up with the energy dissipation measure $D[u]$, which was first described in $(\ref{40})$. They also found the interesting identity $D[u]=\lim\limits_{\kappa \rightarrow\infty} \Pi_{\kappa}[u]$ for distributions, where $u\in L_{t}^{3}L^{3}$ is a weak solution of the Euler equations.
\section{Preliminaries}
In this section, we are going to recall some basic facts on the Littlewood–Paley theory, the
definition of Besov space and some useful lemmas
\begin{defn}
    Let $\mathcal{S}(R^{d})$ be the space of Schwartz class of rapidly decreasing
functions such that for any $k\in \mathbb{N}$
\begin{eqnarray*}
    \lVert u \rVert_{k,\mathcal{S}}=\sup_{\lvert \alpha \rvert\leq k , x\in \mathbb{R}^{d} } (1+\lvert x \rvert)^{k} \lvert \partial^{\alpha} u(x)\rvert \leq \infty
\end{eqnarray*}
\end{defn}
We choose two nonnegative functions $\chi , \varphi \in \mathcal{S}(\mathbb{R}^{n})$  respectively, support in
\begin{eqnarray*}
    \mathcal{B}=\lbrace \xi\in \mathbb{R}^{n}:\lvert \xi \rvert \leq \frac{4}{3} \rbrace\\
\mathcal{C}=\lbrace \xi\in \mathbb{R}^{n}: \frac{4}{3} \leq \lvert \xi \rvert \leq \frac{8}{3} \rbrace
\end{eqnarray*}
such that
\begin{eqnarray*}
    \chi(\xi)+\sum\limits_{j \geq 0} \varphi(2^{-j}\xi)=1 \;\;\mbox{for all}\;\; \xi\in\mathbb{R}^{n}
    \\
    \sum\limits_{j \in \mathbb{Z}} \varphi(2^{-j}\xi)=1 \;\;\mbox{for all}\;\; \xi\in\mathbb{R}^{n}  \lbrace0\rbrace
\end{eqnarray*}
Setting $ \varphi_{j}= \varphi(2^{-j}\xi)$ then
\begin{eqnarray*}
    \begin{cases}
          supp ~\varphi_{j} \cap   supp ~\varphi_{j'}= \phi \;\;\mbox{if}\;\; \lvert j-j' \rvert \geq 2
          \\
           supp ~\chi \cap   supp ~\varphi_{j}= \phi \;\;\mbox{if}\;\; j \geq 1
    \end{cases}
\end{eqnarray*}
Let 
$h=\mathcal{F}^{-1}\varphi$ and $\tilde{h}=\mathcal{F}^{-1} \chi$ . Define the frequency localization operators
\begin{eqnarray*}
    \Delta_{j} u =0 \;\;\mbox{for}\;\; j \leq -2 ~~~~~~~~~~~~~~~ \Delta_{-1}u=S_{0}u=\chi(D)u
    \\
    \Delta_{j} u = \varphi(2^{-j}D)u=2^{nj}\int_{\mathbb{R}^{n}}h(2^{j}y)u(x-y)dy\;\;\mbox{for}\;\; j \geq 0
    \\
    S_{j}u=\chi(2^{-j}D)u=\sum\limits_{-1 \leq k \leq j-1} \Delta_{k}u\int_{\mathbb{R}^{n}}\tilde{h}(2^{j}y)u(x-y)dy
\end{eqnarray*}
Informally $\Delta_{j}=S_{j+1}-S_{j}$ is a frequency projection to the annulus $\lvert \xi \rvert \approx 2^{j}$ , while $S_{j}$ is
the frequency projection to the ball $\lvert \xi \rvert \precsim 2^{j}$. One easily verifies that with the above choice
of $\varphi$
\begin{eqnarray*}
    \begin{cases}
        \Delta_{j}\Delta_{k}u \equiv 0 \;\;\mbox{if}\;\; \lvert j-k\rvert\geq 2
        \\
         \Delta_{j}(S_{k-1}u \Delta_{k}u)\equiv 0 \;\;\mbox{if}\;\; \lvert j-k\rvert\geq 5
    \end{cases}
\end{eqnarray*}
We revisit Bony’s paraproduct decomposition. Let $u$ and $v$ be two temperate distributions,
the paraproducts between $u$ and $v$ are defined by
\begin{eqnarray*}
\begin{cases}
     T_{u}v:=\sum\limits_{j} S_{j-1}u \Delta_{j}v
     \\
        T_{v}u:=\sum\limits_{j} S_{j-1}v \Delta_{j}u
\end{cases}
\end{eqnarray*}
Define the remainder of the paraproduct $R(u,v)$ as
\begin{eqnarray*}
    R(u,v):=\sum\limits_{\lvert j-j^{'}\rvert\leq 1} \Delta_{j}u \Delta_{j'}v
\end{eqnarray*}
Then, we have the following Bony’s decomposition:
\begin{eqnarray*}
    uv=T_{u}v+T_{v}u+R(u,v)
\end{eqnarray*}
We shall sometimes also use the following simplified decomposition
\begin{eqnarray*}
    uv=T_{u}v+T^{'}_{v}u \;\;\mbox{ with } \;\; T^{'}_{v}u =T_{v}u+R(u,v)= \sum\limits_{j} S_{j+2} v \Delta_{j}u
\end{eqnarray*}
Now we introduce the definition of inhomogeneous Besov spaces by means of the
Littlewood–Paley projection $\Delta_{j}$ and $S_{j}$
\begin{defn}
    Let $r \in \mathbb{R} ,~~1 \leq p,q \leq \infty$ , the inhomogeneous Besov space
    \begin{eqnarray*}
        B_{p,q}^{r}(\mathbb{R}^{d}):= \lbrace u \in \mathcal{S'}(R^{n}): \lVert u \rVert_{B_{p,q}^{r}}\leq \infty \rbrace
    \end{eqnarray*}
\end{defn}
where
\begin{eqnarray*}
    \lVert u \rVert_{B_{p,q}^{r}}= \begin{cases}
        \Big(\sum\limits_{j=-1}^{\infty} 2^{jrq} \lVert \Delta_{j} u \rVert_{L^{p}}^{q}\Big)^{\frac{1}{q}}\;\;\mbox{for} \;\; q \leq \infty
        \\
      \sup\limits_{j\geq -1}  2^{js}\lVert \Delta_{j} u\rVert_{L^{p}}\;\;\mbox{for} \;\; q = \infty
    \end{cases}
\end{eqnarray*}
\begin{defn}
We'll employ the subsequent inequality for functions within Besov spaces. for more details see  \cite{10}
\begin{eqnarray*}
     \lVert f(.+\xi)-f(.)||_{L^{p}}\leq C |\xi|^{r}||f||_{B_{p,\infty}^{r}}
\end{eqnarray*}
    which holds for $1\leq p \leq \infty $ , $r>0$ and some constant C \label{def3.3}
\end{defn}
\section{Convergence Dynamics of 3D-Leray  Gaseous stars model}
The model system for viscous incompressible fluids, known as the 3D convergence Dynamics of Gaseous stars, is represented by a specific shape:
\begin{equation*}
\left\{
\begin{aligned}
    &\partial_{t} \textbf{u} + (\textbf{u}\cdot\nabla)\textbf{u} -\nu\Delta\textbf{u}+\nabla p = 0,\quad \quad\quad (x,t)\in \mathbb{T}^{3}\times (0,T)\\
    &\partial_{t}Z + (\textbf{u}\cdot\nabla)Z + K\phi(\Bar{\theta})\Bar{Z}=0,\quad \quad\quad(x,t)\in \mathbb{T}^{3}\times (0,T)\\
&K\phi(\Bar{\theta})\Bar{Z}=0,\quad \quad\quad (x,t)\in \mathbb{T}^{3}\times (0,T)\\
&\nabla \cdot\textbf{u}=0 \quad \quad\quad (x,t)\in \mathbb{T}^{3}\times (0,T)
\end{aligned}
\right.
\end{equation*}
where $\textbf{u}=(u_{1}(x,t),u_{2}(x,t),u_{3}(x,t))$ is the unknown velocity field of a fluid pattern at point $x$ and at time $t$, $p(x,t)$ is the unknown pressure, and $Z(x,t)$ is the mass fraction of the reactant, $\theta(x,t)$ is the temperature, and the positive parameter $\nu$ is the kinematic viscosity of the fluid.
\par Leray regularization in \cite{9} is employed to explain the 3D convergence dynamics of a model system for gaseous stars that has time-varying boundary conditions and viscous fluids that are incompressible.
Assume that the 3-dimensional torus is represented by $x=(x_{1},x_{2},x_{3})\in \mathbb{T}^{3} := ([-\pi,\pi]\big{|}_{ \{\-\pi,\pi \}\ })^{3}$. Then, all functions in each $x_{i}, i=1,2,3$ have a period of $2\pi. $ The vector fields $v=v(x,t)=(v_{1},v_{2},v_{3})$ or $\textbf{u}=(u_{1},u_{2},u_{3})$, as well as the scalar function $p=p(x,t)$, are the unknown functions.
\begin{equation*}
\left\{
\begin{aligned}
    &\partial_{t} \textbf{u} + (v\cdot\nabla)\textbf{u} -\nu\Delta\textbf{u}+\nabla p = 0, 
\quad \quad\quad (x,t)\in \mathbb{T}^{3}\times (0,T)  \\
&\partial_{t}Z + \nabla \cdot ( Z \textbf{u}) +K\phi(\Bar{\theta})\Bar{Z}=0, 
\quad \quad\quad (x,t)\in \mathbb{T}^{3}\times (0,T)\\
&-K\phi(\Bar{\theta})\Bar{Z}=0,
\quad \quad\quad (x,t)\in \mathbb{T}^{3}\times (0,T)\\
 &\textbf{u}= v-\alpha^{2}\Delta v, \nabla \cdot \textbf{u}= \nabla \cdot v =\nabla \cdot Z=0,\quad \quad\quad (x,t)\in \mathbb{T}^{3}\times (0,T)
\end{aligned}
\right.
\end{equation*}
A specific smoothing kernel, the Green function related to the Helmholtz operator, was studied.
The Helmholtz regularization of $\textbf{u}$ is
$$
\textbf{u}=(1-\alpha^{2}\Delta)v.
$$
The fixed positive parameter $\alpha$ is called the model's sub-grid (filter) length scale. The function $\textbf{u}=v$ yields the exact equations for the $3D$ convergence dynamics of the gaseous star model for $\alpha =0$.
\par In the end, this work examines the inviscid variant of the viscous model:
\begin{equation}\label{sys9} 
\left\{
\begin{aligned}
 &\partial_{t} \textbf{u} + (v\cdot\nabla)\textbf{u} +\nabla p = 0,\quad\quad\quad(x,t)\in \mathbb{T}^{3}\times (0,T)\\
&\partial_{t}Z + \nabla \cdot ( Z \textbf{u}) +K\phi(\Bar{\theta})\Bar{Z}=0,\quad\quad\quad(x,t)\in \mathbb{T}^{3}\times (0,T)\\
&-K\phi(\Bar{\theta})\Bar{Z}=0,\quad\quad\quad (x,t)\in \mathbb{T}^{3}\times (0,T)\\
&\textbf{u}= v-\alpha^{2}\Delta v,\; \nabla \cdot \textbf{u}= \nabla \cdot v =\nabla \cdot Z=0.\quad\quad\quad (x,t)\in \mathbb{T}^{3}\times (0,T)
\end{aligned}
\right.
\end{equation}
\section{Distributional solutions}
\begin{defn}
A weak solution of the Leray-alpha -$\alpha$ gaseous stars model over the domain $\mathbb{T}^{3}\times (0,T)$ is given a velocity field 
$\textbf{u}=(u_{1},u_{2},u_{3}): \mathbb{T}^{3}\times(0,T)\rightarrow \mathbb{R}^{3}$ and a pressure $p:\mathbb{T}^{3}\times(0,T)\rightarrow \mathbb{R}$
is a triplet $ \textbf{u}\in L^{\infty}(0,T;(L^{2}(\mathbb{T}^{3}))),p\in L^{\infty}(0,T;L^{1}(\mathbb{T}^{3})),Z \in L^{\infty}(0,T;L^{2}(\mathbb{T}^{3})) $ if for all $ \varphi \in \mathcal{D}(\mathbb{T}^{3}\times(0,T);\mathbb{R}^{3})
$         
and
$ \chi_{1},\chi_{2} \in \mathcal{D}(\mathbb{T}^{3}\times(0,T);\mathbb{R})$ the following equations holds
\begin{eqnarray}\label{sys10}
\begin{aligned}
&\int_{0}^{T} \int_{\mathbb{T}^{3}} \textbf{u}\cdot \partial_{t}\varphi dx dt +\int_{0}^{T} \int_{\mathbb{T}^{3}} \textbf{u}_{j} v_{i} \partial_{i} \varphi_{j}dxdt+\int_{0}^{T} \int_{\mathbb{T}^{3}} p \nabla\cdot\varphi dx dt=0,\\
&\int_{0}^{T} \int_{\mathbb{T}^{3}} Z\cdot \partial_{t}\chi_{1}  dx dt +\int_{0}^{T} \int_{\mathbb{T}^{3}} Z\textbf{u}\cdot\nabla \chi_{1} dxdt+\int_{0}^{T} \int_{\mathbb{T}^{3}} (K\phi(\Bar{\theta})\Bar{Z} )\chi_{1} dx dt=0,\\
&\int_{0}^{T} \int_{\mathbb{T}^{3}} v_{i}\partial_{i}\chi_{2} dx dt =\int_{0}^{T} \int_{\mathbb{T}^{3}} Z_{i}\partial_{i}\chi_{2} dxdt=0
\end{aligned}
\end{eqnarray}
\end{defn}
\begin{prop}(\cite{7})\label{prop 9}
Let $\textbf{u}= (u_{1},u_{2},u_{3}) \in L^{\infty}(0,T;L^{2} (\mathbb{T}^{3})) \cap  L^{2}(0,T;H_{0}^{1} (\mathbb{T}^{3}))$. The 
Sobolev inclusion  $ H_{0}^{1} (\mathbb{T}^{3}) \subset L^{6}  (\mathbb{T}^{3})$  gives  
    $\textbf{u} \in L^{\infty}(0,T;L^{2}({\mathbb{T}^{3}} )) \cap  L^{2}(0,T;L^{6}
({\mathbb{T}^{3}})) $. Furthermore, H\"older's inequality provides that $ u \in  L^{3}([0,T];L^{3}(\mathbb{T}^{3}))$.  ,
Also 
 $ Z \in  L^{3}([0,T];L^{3}(\mathbb{T}^{3}))$. 
\end{prop}
 \begin{lma}\label{lemma 3}
The problem (\ref{sys10}) applies to test functions $\varphi$ , $\chi_{1} \in  W_{0}^{1,1}((0,T);L^{2}(\mathbb{T}^{3}))\cap L^{1}((0,T);H^{3}(\mathbb{T}^{3}))$ 
    \begin{proof}
Here, we approach \cite{9} to prove this Lemma. For every $\varphi\in  W_{0}^{1,1}((0,T);L^{2}(\mathbb{T}^{3}))\cap L^{1}((0,T);H^{3}(\mathbb{T}^{3}))$
there exists a sequence of test function $\varphi_{m}\in \mathcal{D}(\mathbb{T}^{3}\times(0,T))$ that converges to $\varphi$.
\par It is important to note that the problem $(\ref{sys10})$ is valid for any $\varphi_{m}$, as it is contained in $\mathcal{D}(\mathbb{T}^{3}\times(0,T)). $
        Now, note that this problem holds for any $\varphi_{m}$, since they lie in $\mathcal{D}(\mathbb{T}^{3}\times(0,T))$.     
        Also, we observe that 
        $$
				\textbf{u}\partial_{t}\varphi_{m}\rightarrow \textbf{u}\partial_{t}\varphi\;\;\mbox{ and}\;\; Z\partial_{t}\varphi_{m}\rightarrow Z\partial_{t}\varphi \;\;\mbox{in}\;\; L^{1}((0,T)\times \mathbb{T}^{3})\;\;\mbox{as}\;\;m\rightarrow \infty.
				$$
				Thus,
\begin{equation*}
\begin{aligned}
            &\int_{0}^{T} \int_{\mathbb{T}^{3}}\textbf{u}\cdot \partial_{t}\varphi_{m} dxdt \underset{m \rightarrow \infty}{\longrightarrow}\int_{0}^{T}\int_{\mathbb{T}^{3}}\textbf{u}\cdot \partial_{t}\varphi dxdt,\\
            &\int_{0}^{T} \int_{\mathbb{T}^{3}} Z \partial_{t}\varphi_{m} dxdt \underset{m \rightarrow \infty}{\longrightarrow}\int_{0}^{T}\int_{\mathbb{T}^{3}} Z \partial_{t}\varphi dxdt.\\
\end{aligned}
\end{equation*}
        Similarly, we have
        $$
				\textbf{u}_{j}v_{i}\partial_{i}(\varphi_{m})_{j}\rightarrow \textbf{u}_{j}v_{i}\partial_{i}\varphi_{j} \;\;\mbox{and}\;\;Z\textbf{u}\cdot \nabla\varphi_{m}\rightarrow Z\textbf{u}\cdot \nabla\varphi \;\;\mbox{in}\;\;L^{1}((0,T)\times \mathbb{T}^{3})\;\;\mbox{as}\;\;m\rightarrow \infty,
				$$
        which means that
\begin{equation*}
\begin{aligned}
            &\int_{0}^{T} \int_{\mathbb{T}^{3}}\textbf{u}_{j}v_{i}\partial_{i}(\varphi_{m})_{j} dxdt \underset{m \rightarrow \infty}{\longrightarrow}\int_{0}^{T} \int_{\mathbb{T}^{3}}\textbf{u}_{j}v_{i}\partial_{i}\varphi_{j} dxdt\\
            &\int_{0}^{T} \int_{\mathbb{T}^{3}} Z\textbf{u}\cdot \nabla\varphi_{m} dxdt \underset{m \rightarrow \infty}{\longrightarrow}\int_{0}^{T} \int_{\mathbb{T}^{3}} Z\textbf{u}\cdot \nabla\varphi dxdt.
\end{aligned}
\end{equation*} 
Finally, it is important to note that we have made a distinct assumption regarding the pressure's regularity $p\in L^{\infty}(0,T;L^{2}(\mathbb{T}^{3}))$.
Therefore, 
$$
p\partial_{i}\varphi_{m}\rightarrow p\partial_{i}\varphi \;\;\mbox{in}\;\;L^{1}((0,T)\times \mathbb{T}^{3})\;\;\mbox{as}\;\;m\rightarrow \infty.
$$
\begin{eqnarray}
    \int_{0}^{T} \int_{\mathbb{T}^{3}} p\partial_{i}\varphi_{m} dxdt \underset{m \rightarrow \infty}{\longrightarrow}\int_{0}^{T} \int_{\mathbb{T}^{3}} p\partial_{i}\varphi dxdt.
\end{eqnarray}
The convergence of the other terms operates in an identical manner, leading us to conclude proof.
 To fix the constant in the pressure definition, we require that:
  $$
	\int_{\mathbb{T}^{3}} p(x,t) dx =0.
	$$ 
\end{proof}
We introduce $\varphi\in C_{c}^{\infty} (\mathbb{R}^{3};\mathbb{R}) $ radial standard $C_{c}^{\infty}$ mollifier with 
the property that $\int_{\mathbb{R}^{3} }\varphi (x) dx=1$  we define
\begin{eqnarray*}
    \varphi_{\epsilon}(x) := \frac{1}{\epsilon^{3}}\varphi(\frac{x}{\epsilon})
\end{eqnarray*}
Moreover, we introduce the notation
\begin{eqnarray*}
    u^{\epsilon}=u\ast \varphi_{\epsilon}
\end{eqnarray*}
Throughout the paper we will be using the Einstein summation convention
\begin{eqnarray*}
    u_{i}u_{i}=\sum\limits_{j} u_{i}u_{i}
\end{eqnarray*}

\end{lma}
 \section{Energy balance equation}
 Let $(\textbf{u},Z)$ be a weak solution of the convergence dynamics of Leray -$\alpha$ Gaseous stars model such that 
 \begin{eqnarray*}
     \textbf{u}, Z \in L^{3}([0,T];L^{3}(\mathbb{T}^{3})). 
 \end{eqnarray*}
Set
\begin{equation}\label{40}
\begin{aligned}
     D_{1}(v,\textbf{u})(x,t)&:=\lim\limits_{\epsilon\rightarrow 0}   \dfrac{1}{2}\int_{\mathbb{R}^{3}}\nabla_{\xi}\chi_{\epsilon} (\xi)\cdot\delta v(\xi;x,t)|\delta \textbf{u}(\xi;x,t)|^{2}d\xi 
      \\
      & =-\frac{1}{2} \partial_{i}(v_{i}\textbf{u}_{j}\textbf{u}_{j})^{\epsilon} + \frac{1}{2} v_{i} \partial_{i}(\textbf{u}_{j}\textbf{u}_{j})^{\epsilon}+\textbf{u}_{j}\partial_{i}(\textbf{u}_{j}v_{i})^{\epsilon}-v_{i}\textbf{u}_{j}\partial_{i}\textbf{u}_{j}^{\epsilon},
\end{aligned}
\end{equation}
\begin{equation*}
\begin{aligned}
    D_{2}(\textbf{u},Z)(x,t)& :=\lim\limits_{\epsilon\rightarrow 0} \dfrac{1}{2}\int_{\mathbb{R}^{3}}\nabla_{\xi}\chi_{\epsilon} (\xi)\cdot\delta \textbf{u}(\xi;x,t) (|\delta Z(\xi;x,t) |^{2}d\xi\\
    & =-\frac{1}{2}\nabla\cdot(Z^{2}\textbf{u})^{\epsilon} + \frac{1}{2} \textbf{u}\cdot \nabla(Z^{2})^{\epsilon}-Z\textbf{u}\cdot \nabla Z^{\epsilon}+Z\nabla\cdot (Z\textbf{u})^{\epsilon}
\end{aligned}
\end{equation*}
and 
$$\delta f(\xi,x,t)=  f(\xi+x,t)- f(x). $$

 \begin{thm}\label{4.1}  
 The following equation of local energy balance holds in the sense of distribution
 $ \mathcal{D}(\mathbb{T}^{3}\times(0,T))$ 
  $$
\partial_{t}(|\textbf{u}|^{2}+|Z|^{2})+2\nabla\cdot(p\textbf{u})+\nabla\cdot(|\textbf{u}|^{2}v)+\nabla\cdot(|Z|^{2}\textbf{u})+2K\phi(\Bar{\theta})\Bar{Z}Z+ D_{1}(v,\textbf{u})+D_{2}
 (\textbf{u},Z)=0,
$$
such that $ K\phi(\Bar{\theta})\Bar{Z} =0$ as mentioned in the above exact system. Thus, it becomes the final form
 \begin{eqnarray*}
\partial_{t}(|\textbf{u}|^{2}+|Z|^{2})+2\nabla\cdot(p\textbf{u})+\nabla\cdot(|\textbf{u}|^{2}v)+\nabla\cdot(|Z|^{2}\textbf{u})+ D_{1}(v,\textbf{u})+D_{2}
 (\textbf{u},Z)=0.
 \end{eqnarray*}
 \end{thm}
mollifying the convergence dynamics equation of the Leray-$\alpha$- gaseous stars model with $\chi_{\epsilon}$ yields that
\begin{equation}\label{50}
\left\{
\begin{aligned}
&\partial_{t} \textbf{u}^{\epsilon} + \nabla\cdot(v\otimes\textbf{u})^{\epsilon} +\nabla p^{\epsilon} = 0,\\
&\partial_{t}Z^{\epsilon} + \nabla \cdot ( Z \textbf{u})^{\epsilon} + K((\phi(\Bar{\theta})) \Bar{Z})^{\epsilon} =0
\end{aligned}
\right.
\end{equation}
This system holds pointwise in $\mathbb{T}^{3}\times (0,T)$ . we observe that $\textbf{u}^{\epsilon}$,$ Z^{\epsilon}$  $\in L^{\infty}((0,T);C^{\infty}(\mathbb{T}^{3}))$

Regarding the system (\ref{50}), we remark that
$$
\partial_{i}(v_{i}\textbf{u}_{j})^{\epsilon}+ \nabla p^{\epsilon}\in L^{\infty}((0,T);C^{\infty}(\mathbb{T}^{3}))
$$
which implies that $\textbf{u} \in W^{1,\infty}((0,T);C^{\infty}(\mathbb{T}^{3}))$.
Moreover, 
\begin{eqnarray*}
    W_{0}^{1,\infty}((0,T);C^{\infty}(\mathbb{T}^{3})) \subset W_{0}^{1,1}((0,T);C^{\infty}(\mathbb{T}^{3})) \subset W_{0}^{1,1}((0,T);H^{1}(\mathbb{T}^{3})).
\end{eqnarray*}
Also, 
$$
 H^{1}(\mathbb{T}^{3}) \subset L^{2}(\mathbb{T}^{3}).
$$ 
Thus, 
$$
\textbf{u}^{\epsilon}\in W_{0}^{1,1}((0,T);L^{2}(\mathbb{T}^{3})).
$$
The second equation of (\ref{50}) gives $$\nabla \cdot ( Z \textbf{u})^{\epsilon}\in L^{2}((0,T);C^{\infty}(\mathbb{T}^{3})),\;\;K((\phi(\Bar{\theta})) \Bar{Z})^{\epsilon}\in L^{\infty}((0,T);C^{\infty}(\mathbb{T}^{3}))\;\; \mbox{and}\;\;
\nabla p\in L^{\infty}((0,T);C^{\infty}(\mathbb{T}^{3})).
$$ 
So,  
$$\partial_{t}Z^{\epsilon}\in L^{2}((0,T);C^{\infty}(\mathbb{T}^{3})).$$
Consequently,
\begin{equation*}
\begin{aligned}
  Z^{\epsilon}\in  L^{\infty}((0,T);C^{\infty}(\mathbb{T}^{3})) \cap H^{1}((0,T);C^{\infty}(\mathbb{T}^{3})) \subset
  L^{1}((0,T);H^{3}(\mathbb{T}^{3})) \cap
  W^{1,1}((0,T);L^{2}(\mathbb{T}^{3})).
\end{aligned}
\end{equation*}
Subtracting the mollified equation multiplied by $\textbf{u} \chi$ and $Z \chi$ yields that:
\begin{equation*}
\begin{aligned}
	&\int_{0}^{T} \int_{\mathbb{T}^{3}} \textbf{u}\cdot \partial_{t}(\textbf{u}^{\epsilon}\chi)-\textbf{u}\chi \cdot \partial_{t} \textbf{u}^{\epsilon} + Z\cdot \partial_{t}(Z^{\epsilon}\chi)-Z\chi \cdot\partial_{t} Z^{\epsilon} + v\otimes \textbf{u}: \nabla(\textbf{u}^{\epsilon}\chi)-\textbf{u}\cdot(\nabla\cdot(v\otimes\textbf{u})^{\epsilon})+p\partial_{i}(\textbf{u}_{i}^{\epsilon}\chi)\\
	&-\chi\textbf{u}\cdot\nabla p^{\epsilon}+Z\textbf{u}\cdot\nabla(Z^{\epsilon}\chi)- Z\chi \nabla\cdot(Z\textbf{u})^{\epsilon}+K((\varphi(\bar{\theta})\Bar{Z})Z-((\varphi(\bar{\theta}))\Bar{Z})^{\epsilon}Z)\chi dxdt=0.
\end{aligned}
\end{equation*}
The time derivatives then become
\begin{equation*}
\begin{aligned}
&\int_{0}^{T} \int_{\mathbb{T}^{3}} \left[ \textbf{u}\cdot \partial_{t}(\textbf{u}^{\epsilon}\chi)-\textbf{u}\chi \cdot \partial_{t} \textbf{u}^{\epsilon} + Z\cdot \partial_{t}(Z^{\epsilon}\chi)-Z\chi \cdot\partial_{t} Z^{\epsilon}\right] dxdt\\
 &=  \int_{0}^{T} \int_{\mathbb{T}^{3}} \left[\textbf{u}\cdot \textbf{u}^{\epsilon}+Z\cdot Z^{\epsilon}\right]\partial_{t}\chi \mathrm{dx}\mathrm{dt} = - \langle \partial_{t}(\textbf{u}\cdot \textbf{u}^{\epsilon}+Z\cdot Z^{\epsilon}),\chi\rangle.
\end{aligned}
\end{equation*}
For the pressure terms, we get
\begin{equation*}
\begin{aligned}
    &\int_{0}^{T} \int_{\mathbb{T}^{3}}\left[p\partial_{i}(\textbf{u}_{i}^{\epsilon}\chi)
   -\chi\textbf{u}\cdot\nabla p^{\epsilon}\right]dxdt=\int_{0}^{T} \int_{\mathbb{T}^{3}}\left[p^{\epsilon}\textbf{u}_{i}\partial_{i}\chi+p\textbf{u}_{i}^{\epsilon}\partial_{i}\chi\right]dxdt= -\langle \nabla\cdot(p^{\epsilon}\textbf{u}+p\textbf{u}^{\epsilon}),\chi \rangle\\
     &\int_{0}^{T} \int_{\mathbb{T}^{3}}\left[v\otimes \textbf{u}: \nabla(\textbf{u}^{\epsilon}\chi)- \chi\textbf{u}\cdot(\nabla\cdot(v\otimes\textbf{u})^{\epsilon})\right]dxdt 
\end{aligned}
\end{equation*}
But
      $$
			v\otimes \textbf{u}: \nabla(\textbf{u}^{\epsilon}\chi)= \textbf{u}_{j}v_{i} \partial_{i}(\textbf{u}_{j}^{\epsilon}\chi)
			$$
    and
    $$
		\chi\textbf{u}\cdot(\nabla\cdot(v\otimes\textbf{u})^{\epsilon})=\chi \textbf{u}_{j}\partial_{i}(v_{i}\textbf{u}_{j})^{\epsilon}.
		$$
     Thus,
\begin{equation*}
\begin{aligned}
\int_{0}^{T} \int_{\mathbb{T}^{3}}\left[\textbf{u}_{j}v_{i} \partial_{i}(\textbf{u}_{j}^{\epsilon}\chi)-\chi \textbf{u}_{j}\partial_{i}(v_{i}\textbf{u}_{j})^{\epsilon}\right]&dxdt=\int_{0}^{T} \int_{\mathbb{T}^{3}}\left[\chi\textbf{u}_{j}v_{i} \partial_{i}(\textbf{u}_{j}^{\epsilon})+ \textbf{u}_{j}\textbf{u}_{j}^{\epsilon}v_{i}\partial_{i}(\chi)-\chi \textbf{u}_{j}\partial_{i}(v_{i}\textbf{u}_{j})^{\epsilon}\right]dxdt\\
     &= -\int_{0}^{T} \int_{\mathbb{T}^{3}}\left[\chi \textbf{u}_{j}\partial_{i}(v_{i}\textbf{u}_{j})^{\epsilon}-\chi\textbf{u}_{j}v_{i} \partial_{i}(\textbf{u}_{j}^{\epsilon})- \textbf{u}_{j}\textbf{u}_{j}^{\epsilon}v_{i}\partial_{i}(\chi)\right]dxdt\\
     &= -\int_{0}^{T} \int_{\mathbb{T}^{3}}\bigg[\chi D_{1,\epsilon}(v,\textbf{u})+\dfrac{1}{2}((\textbf{u}_{j}\textbf{u}_{j})^{\epsilon}v_{i}-(v_{i}\textbf{u}_{j}\textbf{u}_{j})^{\epsilon})\partial_{i}\chi\\
  & -\textbf{u}_{j}v_{i}\textbf{u}_{j}^{\epsilon}\partial_{i}\chi\bigg]dxdt\\
& =-\langle D_{1,\epsilon}(v,\textbf{u})+\dfrac{1}{2}\nabla\cdot(
    (|\textbf{u}|^{2}v)^{\epsilon}-(|\textbf{u}|^{2})^{\epsilon}v)+\nabla\cdot((\textbf{u}\cdot\textbf{u}^{\epsilon})v),\chi\rangle.
\end{aligned}
\end{equation*}
This allows us to apply
Lemma $\ref{lemma 3}$ and use
$\textbf{u} \chi $ , $ Z \chi$
We find that the weak formulations serve as test functions with respect to the other adjective terms:
\begin{equation*}
\begin{aligned}
    \int_{0}^{T} \int_{\mathbb{T}^{3}} Z\textbf{u}\cdot \nabla(Z^{\epsilon}\chi)&-\chi Z \nabla\cdot(Z\textbf{u})^{\epsilon} dx dt = \int_{0}^{T} \int_{\mathbb{T}^{3}} Z\textbf{u}\chi \cdot \nabla Z^{\epsilon}+ZZ^{\epsilon}\textbf{u}\cdot \nabla \chi -\chi Z \nabla \cdot(Z\textbf{u})^{\epsilon} dx dt\\
&\int_{0}^{T} \int_{\mathbb{T}^{3}} Z\textbf{u}\chi \cdot \nabla Z^{\epsilon}-\chi Z \nabla \cdot(Z\textbf{u})^{\epsilon}+ZZ^{\epsilon}\textbf{u}\cdot \nabla \chi dx dt\\
    &=  \int_{0}^{T} \int_{\mathbb{T}^{3}} -\chi D_{2,\epsilon}(\textbf{u},Z) -\dfrac{1}{2}\chi \nabla\cdot (Z^{2}\textbf{u})^{\epsilon}+\dfrac{1}{2}\chi \textbf{u}\cdot \nabla (Z^{2})^{\epsilon}+ZZ^{\epsilon}\textbf{u}\cdot \nabla \chi dx dt\\
&=  \int_{0}^{T} \int_{\mathbb{T}^{3}}  -\chi D_{2,\epsilon}(\textbf{u},Z) +\dfrac{1}{2}(Z^{2}\textbf{u})^{\epsilon}\cdot\nabla\chi  -\dfrac{1}{2}(Z^{2})^{\epsilon} \textbf{u}\cdot \nabla \chi +ZZ^{\epsilon}\textbf{u}\cdot \nabla \chi dx dt\\
&=\langle -\chi D_{2,\epsilon}(\textbf{u},Z) - \dfrac{1}{2}\nabla \cdot((Z^{2}\textbf{u})^{\epsilon}- (Z^{2})^{\epsilon} \textbf{u} ) -\nabla \cdot (ZZ^{\epsilon}\textbf{u} ) ,\chi\rangle \\
&= - \langle  D_{2,\epsilon}(\textbf{u},Z) + \dfrac{1}{2}\nabla \cdot( Z^{2}\textbf{u})^{\epsilon}-(Z^{2})^{\epsilon} \textbf{u} ) 
    +\nabla \cdot (ZZ^{\epsilon}\textbf{u} ) ,\chi\rangle.  
\end{aligned}
\end{equation*}
Regarding the last terms, we have
\begin{eqnarray*}
    \int_{0}^{T} \int_{\mathbb{T}^{3}} K((\varphi(\bar{\theta})\Bar{Z})Z-((\varphi(\bar{\theta}))\Bar{Z})^{\epsilon}Z)\chi dx dt  = \langle K((\varphi(\bar{\theta})\Bar{Z})Z-((\varphi(\bar{\theta}))\Bar{Z})^{\epsilon}Z), \chi \rangle \\
    =- \langle 
K((\varphi(\bar{\theta}))\Bar{Z})^{\epsilon}Z)-
((\varphi(\bar{\theta})\Bar{Z})Z, \chi \rangle.
\end{eqnarray*}
By combining these findings, we get the following energy expression, appropriate for distributions including test functions in $\mathcal{D}(\mathbb{T}^{3}\times (0,T)). $
 \begin{eqnarray*}
  \Big{\langle}  \partial_{t}(\textbf{u}\cdot \textbf{u}^{\epsilon}+Z\cdot Z^{\epsilon})+ \nabla\cdot(p^{\epsilon}\textbf{u}+p\textbf{u}^{\epsilon}) +D_{1,\epsilon}(v,\textbf{u})+\dfrac{1}{2}\nabla\cdot(
    (|\textbf{u}|^{2}v)^{\epsilon}-(|\textbf{u}|^{2})^{\epsilon}v)+\nabla\cdot((\textbf{u}\cdot\textbf{u}^{\epsilon})v)
    \\ +D_{2,\epsilon}(\textbf{u},Z) + \dfrac{1}{2}\nabla \cdot( 
  Z^{2}\textbf{u})^{\epsilon}-(Z^{2})^{\epsilon} \textbf{u} ) +\nabla \cdot (ZZ^{\epsilon}\textbf{u} ) + K((\varphi(\bar{\theta}))\Bar{Z})^{\epsilon}Z)-
((\varphi(\bar{\theta})\Bar{Z})Z),\chi \Big{\rangle} =0.
\end{eqnarray*}
Now, we consider the convergence of the different terms as $\epsilon\to 0$. since 
$ \textbf{u}, Z \in L^{\infty}([0,T];L^{2}(\mathbb{T}^{3}))$ it holds that 
$$
\textbf{u}\cdot \textbf{u}^{\epsilon}+Z\cdot Z^{\epsilon}\underset{\epsilon\rightarrow 0}{\longrightarrow} 
|\textbf{u}|^{2}+ |Z|^{2} \in L^{\infty}([0,T];L^{1}(\mathbb{T}^{3})). 
 $$ 
So,
 $$
\partial_{t}(\textbf{u}\cdot \textbf{u}^{\epsilon}+Z\cdot Z^{\epsilon})\underset{\epsilon\rightarrow 0}{\longrightarrow} 
\partial_{t}(|\textbf{u}|^{2}+ |Z|^{2}).
$$
But, $\textbf{u} , Z\in L^{3}([0,T];L^{3}(\mathbb{T}^{3}))$. According to Proposition \ref{prop 9}, we obtain $\textbf{u}_{j}v_{i}$ in $L^{\frac{3}{2}}([0,T];L^{\frac{3}{2}}(\mathbb{T}^{3}))$. The divergence of (\ref{sys9} ) gives
\begin{eqnarray*}
    -\Delta p = \partial_{j}\partial_{i}(v_{i}\textbf{u}_{j}).
\end{eqnarray*}
If $p$ represents the sole solution with a zero mean, then the linear operator $    A:v_{i}\textbf{u}_{j} \rightarrow p$ is strongly continuous on $L^{p}$ for $1 < p <\infty$.
Thus, $p\in L^{\frac{3}{2}}([0,T];L^{\frac{3}{2}}(\mathbb{T}^{3}))$ and therefore 
$$p^{\epsilon}\textbf{u}+p\textbf{u}^{\epsilon}\underset{\epsilon\rightarrow 0}{\longrightarrow} 2p\textbf{u} \in L^{1}([0,T];L^{1}(\mathbb{T}^{3}))$$
\begin{equation*}
\begin{aligned}
||p^{\epsilon}\textbf{u}+p\textbf{u}^{\epsilon}-2 p \textbf{u}||_{L^{1}(\mathbb{T}^{3}))}
&=||p^{\epsilon}\textbf{u}- p \textbf{u}+p\textbf{u}^{\epsilon}- p \textbf{u}||_{L^{1}(\mathbb{T}^{3}))}\\
&=||(p^{\epsilon}- p )\textbf{u}+p(\textbf{u}^{\epsilon}-\textbf{u})||_{L^{1}(\mathbb{T}^{3}))}\\
&\leq
||(p^{\epsilon}- p )\textbf{u}||_{L^{1}(\mathbb{T}^{3}))}+||p(\textbf{u}^{\epsilon}-\textbf{u})||_{L^{1}(\mathbb{T}^{3})}\\
&\leq
||(p^{\epsilon}- p )||_{L^{\frac{3}{2}}(\mathbb{T}^{3})}||\textbf{u}||_{L^{3}(\mathbb{T}^{3})}+||p||_{L^{\frac{3}{2}}(\mathbb{T}^{3})}||\textbf{u}^{\epsilon}-\textbf{u}||_{L^{3}(\mathbb{T}^{3})}.
\end{aligned}
\end{equation*}
Proceeding to the limit $\epsilon\rightarrow 0$, we acquire
$$
||(p^{\epsilon}- p )||_{L^{\frac{3}{2}}(\mathbb{T}^{3})}\underset{\epsilon\rightarrow 0}{\longrightarrow} 0\;\;\mbox{and}\;\;||\textbf{u}^{\epsilon}-\textbf{u}||_{L^{3}(\mathbb{T}^{3})}\underset{\epsilon\rightarrow 0}{\longrightarrow} 0.
$$
Thus
$$
||p^{\epsilon}\textbf{u}+p\textbf{u}^{\epsilon}-2 p \textbf{u}||_{L^{1}(\mathbb{T}^{3})}\underset{\epsilon\rightarrow 0}{\longrightarrow} 0.
$$
Now, we show that 
$
(\textbf{u}_{j}\textbf{u}_{j} v_{i})^{\epsilon}-
(\textbf{u}_{j}\textbf{u}_{j})^{\epsilon} v_{i}
\underset{\epsilon\rightarrow 0}{\longrightarrow} 0\;\;\mbox{in}\;\;L^{\infty}([0,T];L^{1}(\mathbb{T}^{3})).$ In fact
\begin{equation*}
\begin{aligned}
 ||(\textbf{u}_{j}\textbf{u}_{j} v_{i})^{\epsilon}-
(\textbf{u}_{j}\textbf{u}_{j})^{\epsilon} v_{i}||_{L^{1}}
&=  ||(\textbf{u}_{j}\textbf{u}_{j} v_{i})^{\epsilon}-\textbf{u}_{j}\textbf{u}_{j} v_{i}
+\textbf{u}_{j}\textbf{u}_{j} v_{i}-
(\textbf{u}_{j}\textbf{u}_{j})^{\epsilon} v_{i}||_{L^{1}}\\
&\leq
||(\textbf{u}_{j}\textbf{u}_{j} v_{i})^{\epsilon}-\textbf{u}_{j}\textbf{u}_{j} v_{i}||_{L^{1}}
+|| v_{i}(\textbf{u}_{j}\textbf{u}_{j}-
(\textbf{u}_{j}\textbf{u}_{j})^{\epsilon} )||_{L^{1}}\\
&\leq
||(\textbf{u}_{j}\textbf{u}_{j} v_{i})^{\epsilon}-\textbf{u}_{j}\textbf{u}_{j} v_{i}||_{L^{1}}+|| v_{i}||_{L^{\infty}}||(\textbf{u}_{j}\textbf{u}_{j})-
(\textbf{u}_{j}\textbf{u}_{j})^{\epsilon} )||_{L^{1}}.
\end{aligned}
\end{equation*}
Once more, 
\begin{eqnarray*}
    ZZ^{\epsilon}   \underset{\epsilon\rightarrow 0}{\longrightarrow}  Z^{2} \in L^{\frac{3}{2}}([0,T];L^{\frac{3}{2}}(\mathbb{T}^{3})).
\end{eqnarray*}
More specifically, 
\begin{eqnarray*}
(ZZ^{\epsilon}\textbf{u})   \underset{\epsilon\rightarrow 0}{\longrightarrow}  (Z^{2}) \textbf{u}\in L^{1}([0,T];L^{1}(\mathbb{T}^{3})).
\end{eqnarray*}
This leads us to conclude that 
\begin{eqnarray*}
(Z^{2}\textbf{u})^{\epsilon} -  (Z^{2})^{\epsilon} \textbf{u} \underset{\epsilon\rightarrow 0}{\longrightarrow} 0 \in L^{1}([0,T];L^{1}(\mathbb{T}^{3}))
\end{eqnarray*}
Hence
\begin{eqnarray*}
    K((\varphi(\bar{\theta}))\Bar{Z})^{\epsilon}Z)-
((\varphi(\bar{\theta})\Bar{Z})Z)\underset{\epsilon\rightarrow 0}{\longrightarrow} 0 \in L^{1}([0,T];L^{1}(\mathbb{T}^{3})).
\end{eqnarray*}
\begin{prop}\label{prop 4}
   Let $\textbf{u}, Z $ be a weak solution of the the Leray-$\alpha $ Gaseous stars model.
   Let $C \in L^{1}(0,T)$ and assume that $\sigma_{i}\in L_{loc}^{\infty}(\mathbb{R})$ such that $\sigma_{i}(|\xi|)\rightarrow 0$ as $(|\xi|)\rightarrow 0 $ for $i=1,2$. Moreover, assume that  
\begin{equation*}
\begin{aligned}
\int_{\mathbb{T}^{3}}|\delta v(\xi,x,t)|(|\delta \textbf{u}(\xi,x,t)|^{2})dx &\leq& C(t)|\xi|\sigma_{1}(|\xi|) 
&\\\int_{\mathbb{T}^{3}}|\delta \textbf{u}(\xi,x,t)|(|\delta Z(\xi,x,t)|^{2})dx &\leq& C(t)|\xi|\sigma_{2}(|\xi|).
\end{aligned}
\end{equation*}
Then 
$$
\lim\limits_{\epsilon \rightarrow0} D_{1,\epsilon}(v,\textbf{u})=D_{1}(v,\textbf{u})=0\;\;\mbox{and}\;\;\lim\limits_{\epsilon \rightarrow0} D_{2,\epsilon}(\textbf{u},Z)=D_{2}(\textbf{u},Z)=0.
 $$ 
This indicates the weak solution conserves energy.
\end{prop}
\begin{proof}
We have
\begin{equation*}
\begin{aligned}
     |D_{1,\epsilon}(v,\textbf{u})|&=  \dfrac{1}{2}\Big{|}\int_{\mathbb{R}^{3}}\nabla_{\xi}\chi_{\epsilon} (\xi)\cdot\delta v(\xi;x,t)|\delta \textbf{u}(\xi;x,t)|^{2}d\xi\Big{|}\\
     &\leq\dfrac{1}{2}\int_{\mathbb{R}^{3}}|\nabla_{\xi}\chi_{\epsilon} (\xi)||\delta v(\xi;x,t)||\delta \textbf{u}(\xi;x,t)|^{2}d\xi,
\end{aligned}
\end{equation*}
\begin{equation*}
\begin{aligned}
       | D_{2}(\textbf{u},Z)|&= \dfrac{1}{2}\Big{|}\int_{\mathbb{R}^{3}}\nabla_{\xi}\chi_{\epsilon} (\xi)\cdot\delta \textbf{u}(\xi;x,t) (|\delta Z(\xi;x,t) |^{2}d\xi\Big{|}\\
			&\leq
        \dfrac{1}{2}\int_{\mathbb{R}^{3}}|\nabla_{\xi}\chi_{\epsilon} (\xi)||\delta \textbf{u}(\xi;x,t)||\delta Z(\xi;x,t)|^{2}d\xi.
\end{aligned}
\end{equation*}
Integrating this inequality over $\mathbb{T}^{3}\times(0,T)$ implies that
\begin{equation*}
\begin{aligned}
   \int_{0}^{T} \int_{\mathbb{T}^{3}} |D_{1,\epsilon}(v,\textbf{u})|dx dt &=  
     \dfrac{1}{2} \int_{0}^{T} dt \int_{\mathbb{R}^{3}}|\nabla_{\xi}\chi_{\epsilon} (\xi)||\int_{\mathbb{T}^{3}}|\delta v(\xi;x,t)||\delta \textbf{u}(\xi;x,t)|^{2} dx d\xi\\
      &\leq   \dfrac{1}{2} \int_{0}^{T} C(t) dt \int_{\mathbb{R}^{3}} \dfrac{1}{\epsilon^{3}}|\nabla_{\xi}\chi(\frac{\xi}{\epsilon})||\xi||\sigma_{1}(|\xi|)d\xi
			\end{aligned}
       \end{equation*}
\begin{equation*}
\begin{aligned}
   \int_{0}^{T} \int_{\mathbb{T}^{3}} |D_{2,\epsilon}(\textbf{u},Z)|dx dt &=  
     \dfrac{1}{2} \int_{0}^{T} dt \int_{\mathbb{R}^{3}}|\nabla_{\xi}\chi_{\epsilon} (\xi)||\int_{\mathbb{T}^{3}}|\delta \textbf{u}(\xi;x,t)||\delta Z(\xi;x,t)|^{2} dx d\xi\\
   &\leq   \dfrac{1}{2} \int_{0}^{T} C(t) dt \int_{\mathbb{R}^{3}} \dfrac{1}{\epsilon^{3}}|\nabla_{\xi}\chi(\frac{\xi}{\epsilon})||\xi||\sigma_{2}(|\xi|)d\xi.
	\end{aligned}
       \end{equation*}
The fact that $\int_{0}^{T} C(t) dt < \infty$ and the change of variable $\xi= \epsilon y$ lead to the conclusion that
\begin{equation*}
\begin{aligned}
&\int_{0}^{T} \int_{\mathbb{T}^{3}} |D_{1,\epsilon}(v,\textbf{u})|dx dt \precsim\int_{\mathbb{T}^{3}}|\nabla_{\xi}\chi(y)||\xi||\sigma_{1}(|\xi|)dy=\int_{\mathbb{T}^{3}}|\nabla_{y}\chi(y)||y||\sigma_{1}(\epsilon|y|)dy\\
& \int_{0}^{T} \int_{\mathbb{T}^{3}} |D_{2,\epsilon}(\textbf{u},Z)|dx dt \precsim\int_{\mathbb{T}^{3}}|\nabla_{\xi}\chi(y)||\xi||\sigma_{2}(|\xi|)dy=\int_{\mathbb{T}^{3}}|\nabla_{y}\chi(y)||y||\sigma_{2}(\epsilon|y|)dy
\end{aligned}
\end{equation*}
Since 
$$
\sigma_{i}(|\xi|)\rightarrow 0 \;\;\mbox{for}\;\; i=1,2\;\; \mbox{as}\;\;|\xi|\rightarrow 0
$$
we get 
$$
D_{1,\epsilon}(v,\textbf{u})\rightarrow 0\;\;\mbox{and} D_{2,\epsilon}(\textbf{u},Z)\rightarrow 0\;\;\mbox{in}\;\;L^{1}(\mathbb{T}^{3}\times (0,T)) \;\;\mbox{in}\;\;
\xi \rightarrow 0.
$$
By Lebesgue dominated convergence theorem \cite{18}, we obtain 
$$
D_{1}(v,\textbf{u})=0 \;\;\mbox{and}\;\;D_{2}(\textbf{u},Z)=0.
$$
\end{proof}
\begin{prop}\label{prop 13}
    Let $(\textbf{u},Z)$ be a weak solution of the inviscid convergence dynamics of 3D-Leray-$\alpha $ Gaseous stars model such that 
		$$
		\textbf{u}\in L^{3}((0,T);B_{3,\infty}^{s}(\mathbb{T}^{3}))\;\;\mbox{ and}\;\; Z\in L^{3}((0,T);B_{3,\infty}^{r}(\mathbb{T}^{3}))\;\;\mbox{with}\;\;s,r> 0\;\;\mbox{ and}\;\;s+2r> 1.
		$$
			Then 
			$$
\lim\limits_{\epsilon \rightarrow 0} D_{1,\epsilon}(v,\textbf{u})=D_{1}(v,\textbf{u})=0\;\;\mbox{ and}\;\;\lim\limits_{\epsilon \rightarrow 0} D_{2,\epsilon}(\textbf{u},Z)=D_{2}(\textbf{u},Z)=0,
$$
which implies conservation of energy .
\end{prop}
\begin{proof}
 we will use definition 3.3 for p=3
    \begin{eqnarray*}
        \lVert f(.+\xi)-f(.)\rVert_{L^{3}}\leq C |\xi|^{r}\lVert f\rVert_{B_{3,\infty}^{r}}
    \end{eqnarray*}
    \\
Proposition \ref{prop 4} gives
\begin{equation*}
\begin{aligned}
&\int_{\mathbb{T}^{3}}|\delta v||\delta \textbf{u}|^{2}dx \leq|\xi|^{1+2r}||\textbf{u}||_{B_{3,\infty}^{s}}^{3},\\
    \\
   &\int_{\mathbb{T}^{3}}|\delta \textbf{u}||\delta Z |^{2}dx \leq|\xi|^{s+2r}||\textbf{u}||_{B_{3,\infty}^{s}}||Z||_{B_{3,\infty}^{s}}^{2}.
\end{aligned}
\end{equation*}
Then
$$
\sigma_{1}(|\xi|):=|\xi|^{2r}\rightarrow 0\;\;\mbox{and}\;\;\sigma_{2}(|\xi|):= |\xi|^{s+2r-1}\rightarrow 0\;\;\mbox{as}\;\; |\xi|\rightarrow 0.
$$ 
Therefore $D_{2}(\textbf{u},Z)= D_{1}(v,\textbf{u})=0 $.proposition $\ref{prop 4}$ allows for the completion of the proof. 
\end{proof}
\section{Conservation of energy}
\begin{thm}
     Let $(\textbf{u},Z)$ be a weak solution of the inviscid convergence Dynamics of 3D-Leray-$\alpha $ Gaseous stars model such that $\textbf{u}\in L^{3}((0,T);B_{3,\infty}^{s}(\mathbb{T}^{3}))$ and $ Z\in L^{3}((0,T);B_{3,\infty}^{r}(\mathbb{T}^{3}))$ with $s,r> 0$ and $s+2r> 1$ the weak solution conserves energy so for almost all
     $t_{1},t_{2} \in (0,T)$ it holds that 
     \begin{eqnarray*}
         ||\textbf{u}(t_{1},.)||_{L^{2}}^{2}+ ||Z(t_{1},.)||_{L^{2}}^{2} =  ||\textbf{u}(t_{2},.)||_{L^{2}}^{2} +  ||Z(t_{2},.)||_{L^{2}}^{2}. 
     \end{eqnarray*}
		\end{thm}
   \begin{proof}
     \par  In the same way as in \cite{9}, the local equation of energy established in the Theorem  \ref{4.1} is combined with proposition $\ref{prop 13} $ to produce the following: 
         \begin{eqnarray*}
             \partial_{t}(|\textbf{u}|^{2}+|Z|^{2})+2\nabla\cdot(p\textbf{u})+\nabla\cdot(|\textbf{u}|^{2}v+|Z|^{2}\textbf{u})=0.
         \end{eqnarray*}
         This equation is valid in the distributional sense, meaning that it applies to any test function $\phi \in \mathcal{D}(\mathbb{T}^{3}\times(0,T);\mathbb{T}^{3}).$ So,
\begin{eqnarray}\label{96}
  \int_{0}^{T} \int_{\mathbb{T}^{3}} \frac{1}{2}( |\textbf{u}|^{2}+|Z|^{2})\partial_{t}\phi  dx dt = -  \int_{0}^{T} \int_{\mathbb{T}^{3}} \nabla\phi\cdot (\frac{1}{2}(|\textbf{u}|^{2}v+|Z|^{2}\textbf{u})+ \textbf{u}p)dxdt.
\end{eqnarray}
We want to show that the $L^{2}$ norm of $\textbf{u},Z$ at time $t_{1}$ equals the $L^{2}$ norm at time $t_{2}$ (where $t_{1}<t_{2})$. To this end, let $\phi : \mathbb{R} \rightarrow \mathbb{R}$  be a standard  $C_{c}^{\infty}$ mollified with $\int_{\mathbb{R}} \phi(t)dt=1$ with support contained in $[-1;1]$. 
\par Let us introduce the notation $\phi_{\epsilon}(t):= \frac{1}{\epsilon}\phi(\frac{t}{\epsilon})$
We choose the following test function:
\begin{eqnarray*}
    \phi_{1}(t)=\int_{0}^{t} \phi_{\epsilon}(t'-t_{1})-\phi_{\epsilon}(t'-t_{2}) dt'.
\end{eqnarray*}
For a small enough value of $\epsilon$, we have 
\begin{eqnarray*}
    \begin{cases}
        \phi_{1}(t)=0 \;\;\mbox{if} \;\;t\in ]0,t_{1}-\epsilon[\cup]t_{2}+\epsilon,T[, \\
        \phi_{1}(t)=1 \;\;\mbox{if} \;\;t\in ]t_{1}+\epsilon[\cup]t_{2}-\epsilon[. \\
    \end{cases}
\end{eqnarray*}
This function has compact support in $(0,T)$. Therefore,
\begin{eqnarray*}
    \int_{t_{1}-\epsilon}^{t_{1}+\epsilon} \int_{\mathbb{T}^{3}} (|\textbf{u}|^{2}+|Z|^{2})\phi_{\epsilon}(t-t_{1})dxdt= \int_{t_{2}-\epsilon}^{t_{2}+\epsilon} \int_{\mathbb{T}^{3}} (|\textbf{u}|^{2}+|Z|^{2})\phi_{\epsilon}(t-t_{2})dxdt
\end{eqnarray*}
By Lebesgue differentiation Theorem, we obtain
\begin{eqnarray*}
     \int_{\mathbb{T}^{3}} (|\textbf{u}(x,t_{1})|^{2}+|Z(x,t_{1})|^{2})dx= \int_{\mathbb{T}^{3}} (|\textbf{u}(x,t_{2})|^{2}+|Z(x,t_{2})|^{2})dx\quad\mbox{for all}\;\; t_{1},t_{2} \in (0,T).
\end{eqnarray*}
So, $||\textbf{u}(t,.)||_{L^{2}}^{2}+ ||Z(t,.)||_{L^{2}}^{2},\;\;  t\in (0,T)$ is conserved. Therefore, under the assumptions outlined in Proposition, we have demonstrated the conservation of energy for weak solutions, corresponding to the analog of the first half of Onsager's conjecture for this model.
\end{proof}
\section{Conclusion}
In this work, we examined $3D$-Leray-$\alpha $ Gaseous stars model. Regarding the conserved quantity $||\textbf{u}||_{L^{2}}^{2}+ ||Z||_{L^{2}}^{2}$ we employed $\textbf{u},Z\in L_{t}^{\infty}(L_{x}^{2})$ as a regularity assumption and $Z\in L^{3}((0,T);B_{3,\infty}^{r}(\mathbb{T}^{3})),\;s,r> 0\;\mbox{and}\; s+2r> 1$ as a Besov assumption.
\section*{Acknowledgements}
The authors thank the anonymous referees for their constructive criticism and suggestions.

\end{document}